\documentclass[12pt]{amsart}
\usepackage{amsfonts}

\newtheorem{theorem}{Theorem}[section]
\newtheorem{corollary}[theorem]{Corollary}

\newtheorem{proposition}[theorem]{Proposition}

\newtheorem{remark}[theorem]{Remark}

\def\NN{\hbox{\sf I\kern-.13em\hbox{N}}}
\def\RR{\hbox{\sf I\kern-.14em\hbox{R}}}

\def\Cc{\hbox{\sf C\kern -.47em {\raise .48ex \hbox{$\scriptscriptstyle |$}}
   \kern-.5em {\raise .48ex \hbox{$\scriptscriptstyle |$}} }}

\newcommand{\be}{\begin{equation}}
\newcommand{\ee}{\end{equation}}

\begin{document}

\baselineskip 7mm

\title[Logarithmic convexity of fixed points ]
{Logarithmic convexity of fixed points of stochastic kernel operators}

\author{Aljo\v{s}a Peperko}
\date{\thanks{} \today}

\begin{abstract}
\baselineskip 7mm

In this paper we prove results on logaritmic convexity of fixed points of stochastic kernel operators.
These results are expected to play a key role in the economic application to strategic market games.
\end{abstract}

\maketitle

\noindent
{\it Math. Subj.  Classification (2000)}: 47B34, 15B51, 91B02, 47B65, 15B48, 15A42, 47A10 \\
{\it Key words}: positive kernel operators, stochastic operators, eigenfunctions, non-negative matrices, mathematical economics.  \\
\section{Introduction}

In this paper we prove results on logaritmic convexity of fixed points of stochastic kernel operators.
These result generalize and extend the finite dimensional results from \cite{S93}. We extend the results of \cite{S93}  even in the case of $n \times n$ matrices.  These results were originally motivated by 
economic considerations with application to strategic market games  (see e.g. \cite{SY89}, \cite{ASSY90}, \cite{S10} and the references cited there). More precisely, they were, together with Kakutani's fixed point theorem (see e.g. \cite{K41}), a crucial step in the proof of the existence of Nash equilibria in the model of
an exchange economy with complete markets studied in \cite{SY89}. Thus our results are expected to be a key ingredient in the development of the infinite dimensional generalization of this economic model.

The paper is organized as follows. In the rest of the current section we recall basic definitions and facts, which we will need in our proofs and we prove our main results in Section 2 (Theorem \ref{sahi}). In Section 3 we apply our results to finite or infinite non-negative matrices that 
define weighted operators on sequence spaces (Corollary \ref{sahi_weight}) and explain how our results fit into the economic setting of strategic market games (Remark \ref{Aleluja}(ii)). In Remarks \ref{Aleluja}(i) and (iii) we also
 point out possible applications to Arrow-Debreu model and to  open Leontief model of an economy and to Google Page-rank model of internet usage. For these particular applications it might suffice, if we restricted our results to the setting of Section 3, i.e., to (infinite dimensional) non-negative matrices. However, it is well-known that kernel (integral) operators play a very important, often even central, role in a variety of applications from differential and integro-differential equations, problems from physics (in particular from thermodinamics), engineering, statistical and economic models, etc. (see e.g. \cite{J82}, \cite{R00}, \cite{BP03}, \cite{LL05}, \cite{DLR13}, \cite{P16},  \cite{W75}, \cite{O12} 
and the references cited there). Therefore we choose to present our results in this more general setting. 

Let $\mu$ be a $\sigma$-finite positive measure on a $\sigma$-algebra
$\mathcal{M}$ of subsets of a non-void set $X$. Let $M(X,\mu)_+$ be the cone of 
 all equivalence classes of (almost everywhere equal) $\mu$-measurable functions on $X$
whose values lie in $[0, \infty]$.  For  $f\in M(X,\mu)_+$ we write $f>0$ if and only if $\mu \{x\in X:f(x)>0\}>0$, and $f\gg 0$ if and only if $f(x)>0$ for almost all $x \in X$.

Let $M_0 (X,\mu)$ be the vector lattice of all equivalence classes of (almost everywhere equal) complex $\mu$-measurable functions on $X$.
A vector subspace $L \subseteq M_0 (X,\mu)$ is called an {\it ideal} if $f \in M_0 (X,\mu)$, $g \in L$ 
and $|f| \le |g|$ a.e. imply that $f \in L$. A subset $C \subset L$ is called a {\it wedge}, if   $f+g \in C$ and $\alpha f \in C$ for all $f, g \in C$ and $\alpha > 0$. 
For an ideal $L \subseteq M_0 (X,\mu)$ we say that $X$ {\it is the carrier of} $L$ if there is no subset $Y$ of $X$ of 
strictly positive measure with the property that $f = 0$ a.e. on $Y$ for all $f \in L$ (see \cite{Za83}).

A seminorm $\|\cdot\|$ on the ideal $L \subseteq M_0 (X,\mu)$ is called a 
{\it lattice seminorm} (also Riesz seminorm) if $f \in M_0 (X,\mu)$, $g \in L$ 
and $|f| \le |g|$ a.e. imply that $\|f\| \le \|g\|$. A {\it lattice norm} is a lattice seminorm which is also a norm. 
Recall that an ideal $L \subseteq M_0 (X,\mu)$ equipped with a lattice norm $\|\cdot\|$ is sometimes called a 
{\it normed K\"{o}the space} (\cite[p. 421]{Za83}) 
and that a complete normed K\"{o}the space is called a {\it Banach function space}. 
 Standard examples of Banach function spaces  are Euclidean spaces, the space $c_0$ 
of all null convergent sequences (equipped with the usual norms and the counting measure), the well known $L_p (X, \mu)$ spaces ($1 \le p \le \infty$) and other less known examples such as Orlicz, Lorentz,  Marcinkiewicz  and more general  rearrangement-invariant spaces (see e.g. \cite{BS88}, \cite{CR07} and the references cited there), which are important e.g. in interpolation theory. The set of all normed K\"{o}the spaces (or of all Banach function spaces)  is also closed under all cartesian products $L= E\times F$  equipped with the norm
$\|(x, y)\|_L=\max \{\|x\|_E, \|y\|_F\}$.

The cone of non-negative elements in $L$ is denoted by $L_{+}$. 
By an {\it operator}  from an ideal $L$ to an ideal $N$ we always 
mean a linear operator from $L$ to $N$. An operator $T:L \to N$ is said to be {\it positive} 
if $T f \in N_+$ for all $f \in L_+$. Given operators $S$ and $T$ from $L$ to $N$,
we write $T \le S$ if the operator $S - T$ is positive. A positive operator $T:L \to N$ is called strictly positive if $Tf\gg 0$ whenever  $f\gg 0$, $f\in L_+$.
Recall that a positive operator  between  a  K\"{o}the space $L$ and a Banach function space $N$ is always bounded (see \cite{AA02}, \cite{AB85}), i.e, its operator norm is finite:
\vspace{-4mm}
$$\|A\|=\sup\{\|Af\|_N : f\in L, \|f\|_L \le 1\}=\sup\{\|Af\|_N : f\in L_+, \|f\|_L \le 1\}< \infty.$$

Let $(X,\mu)$ and $(Y, \nu)$ be $\sigma$-finite measure spaces, $L$ and $N$ ideals in 
$M_0(Y,\nu)$ and $M_0(X,\mu)$ respectively, such that $Y$ is the carrier of $L$. An operator $K:L\to N$ is called a {\it kernel operator} if
there exists a $\mu \times \nu$ -measurable function $k(x,y)$ on $X \times Y$ such that, for all $f \in L$ and for almost all $x \in X$,
$ \int_X |k(x,y) f(y)| \, d\nu(y) < \infty$ and $(K f)(x) = \int_X k(x,y) f(y) \, d\nu(y) $.
One can check that a kernel operator $K$ is positive iff its kernel $k$ is non-negative almost everywhere (see \cite{Za83}). 

For the theory of normed K\"{o}the spaces, Banach function spaces, Banach lattices, cones, wedges, positive operators and applications e.g. in financial mathematics we refer the reader 
to the books  \cite{AA02}, \cite{AT07}, \cite{Za83}, 
 \cite{AB85}, \cite{ABB90}, \cite{BS88} and the references cited there. 

Let $L \subset M_0(X,\mu)$ be an ideal. A positive kernel operator $S$ on $L$ with kernel $s(x,y)$ is called
 substochastic, if  $s(y)=\int _X s(x,y)\;dx \le 1$ for almost all $y \in X$. The operator $S$ is called stochastic, if $s(y)=1$ for almost all $y\in X$, and $S$ is called strictly substochastic if $s(y)<1$ for almost all $y\in X$. Note that if $S$ is a stochastic operator on $L^1(X,\mu)$, then $S$ is a Markov operator (see e.g. \cite{R00}). 

\section{Main results}

Firstly, let us recall a version of a result from \cite[Corollary 3.2]{P06}, which follows from the sharpened version of Young's inequality
\be
x^{\alpha}y^{1-\alpha}=\inf _{t>0} \left\{\alpha \,t^{\frac{1}{\alpha}}x + (1-\alpha)t^{-\frac{1}{1-\alpha}}y \right\},
\label{sharpY}
\ee
where  $x,y \ge 0$ and $\alpha \in (0,1)$.
\begin{proposition}
Let $L \subset M_0 (X,\mu)$ be an ideal and $\alpha _i > 0$, $i=1,\dots, m$, such that $\sum _{i=1} ^m \alpha _i =1$.
 Then $f_1 ^{\alpha _1}f_2 ^{\alpha _2}\cdots f_m ^{\alpha _m} \in L$ and 
\be 
\|f_1 ^{\alpha _1}f_2 ^{\alpha _2}\cdots f_m ^{\alpha _m}\|_L\le \|f_1\|_L ^{\alpha _1}\|f_2\|_L ^{\alpha _2}\cdots \|f_m\|_L ^{\alpha _m}
\label{rho} 
\ee
for all $f_i \in L$ and any lattice seminorm $\|\cdot\|_L$.
\label{Riesz}
\end{proposition}

In particular, for $\alpha _i > 0$, $i=1,\dots, m$, such that $\sum _{i=1} ^m \alpha _i = 1$, the following well-known
generalized H\"{o}lder's inequality holds:
\be
 \int \! f_1^{\alpha_1} \, f_2^{\alpha_2} \cdots f_m^{\alpha_m} \, d\mu \le 
   \left( \int \! f_1 \, d\mu \right)^{\alpha_1} \left( \int \! f_2 \, d\mu \right)^{\alpha_2}
   \cdots  \left( \int \! f_m \, d\mu \right)^{\alpha_m} , 
\label{gHold}
\ee
where $f_1, \ldots, f_m$ are non-negative measurable functions on $X$.

The following result follows easily. 
\begin{proposition} Let $L$ and $N$ be ideals in $M_0(Y,\nu)$ and $M_0(X,\mu)$ respectively, such that $Y$ is the carrier of $L$. If $K:L \to N$ is a positive kernel operator and $\alpha _i > 0$, $i=1,\dots, m$, such that $\sum _{i=1} ^m \alpha _i =1$, then 
\be
K(f_1 ^{\alpha _1}f_2 ^{\alpha _2} \cdots f_m ^{\alpha _m}) \le (Kf_1 )^{\alpha _1}(Kf_2) ^{\alpha _2} \cdots (Kf_m) ^{\alpha _m} 
\label{Kalpha}
\ee
for all $f_1, \ldots , f_m \in L_+$.

If, in addition, $\|\cdot\|_N $ is a lattice seminorm on $N$, then
\be
\|K(f_1 ^{\alpha _1}f_2 ^{\alpha _2} \cdots f_m ^{\alpha _m})\|_N \le \|(Kf_1 )^{\alpha _1}(Kf_2) ^{\alpha _2} \cdots (Kf_m) ^{\alpha _m}\|_N  \le \|Kf_1 \| _N ^{\alpha _1}\|Kf_2 \| _N ^{\alpha _2} \cdots \|Kf_m\|_N  ^{\alpha _m}. 
\label{Kseminorm}
\ee
\label{alfe}
\end{proposition}
\begin{proof} For almost all $x\in X$ we have by (\ref{gHold})
$$K(f_1 ^{\alpha _1} \cdots f_m ^{\alpha _m})(x)=  \int_X k(x,y) f_1 ^{\alpha _1}(y) \cdots f_m ^{\alpha _m}(y) \, d\nu(y) $$
$$=  \int_X (k(x,y) f_1 (y)) ^{\alpha _1} \cdots (k(x,y)f_m(y)) ^{\alpha _m} \, d\nu(y) \le  (Kf_1 (x) )^{\alpha _1}  \cdots (Kf_m (x)) ^{\alpha _m},  $$
which proves (\ref{Kalpha}).

The inequalities (\ref{Kseminorm}) now follow from  (\ref{Kalpha}) and (\ref{rho}).
\end{proof}

Let $0<\alpha <1$ and $f_i , g_i \in M(X,\mu)_+$ for $i=1,\ldots ,m$. As already noted (but applied in a different way) in \cite[Inequality (15)]{P06}
an application of 
H\"{o}lder's inequality or (\ref{sharpY}) gives
\be
f_1 ^{\alpha}g_1 ^{1-\alpha} + \ldots +f_m ^{\alpha}g_m ^{1-\alpha} \le (f_1+ \ldots + f_m)^{\alpha}(g_1+ \ldots + g_m)^{1-\alpha}.
\label{H}
\ee
This implies the following result.
\begin{proposition}  Let $L$ and $N$ be ideals in $M_0(Y,\nu)$ and $M_0(X,\mu)$ respectively, such that $Y$ is the carrier of $L$. If $K:L \to N$ is a positive kernel operator and $\alpha _i > 0$, $i=1,\dots, m$, such that $\sum _{i=1} ^m \alpha _i =1$, then 
$$K(f_1 ^{\alpha}g_1 ^{1-\alpha} + \ldots +f_m ^{\alpha}g_m ^{1-\alpha}) \le K\left((f_1+ \ldots + f_m)^{\alpha}(g_1+ \ldots + g_m)^{1-\alpha}\right) $$
\be
 \;\;\;  \;\;\;  \;\;\;  \;\;\;  \;\;\;  \;\;\;  \;\;\;  \;\;\;  \;\;\;  \;\;\;  \;\;\;  \;\;\;  \;\;\;  \;\;\;  \;\;\;  \;\;\le (K(f_1+ \ldots + f_m))^{\alpha}(K(g_1+ \ldots + g_m))^{1-\alpha}.
\label{cool1}
\ee
for all $f_i , g_i \in L_+$, $i=1,\ldots ,m$.

If, in addition, $\|\cdot\|_N $ is a lattice seminorm on $N$, then
$$\|K(f_1 ^{\alpha}g_1 ^{1-\alpha} + \ldots +f_m ^{\alpha}g_m ^{1-\alpha}) \|_N  \le\| K\left((f_1+ \ldots + f_m)^{\alpha}(g_1+ \ldots + g_m)^{1-\alpha}\right)\|_N $$
\be
\le  \|K(f_1+ \ldots + f_m)\|_N ^{\alpha}\|K(g_1+ \ldots + g_m)\|_N ^{1-\alpha}.
\label{cool2}
\ee
\label{+alpha}
\end{proposition}
\begin{proof} The inequalities (\ref{cool1}) follow from (\ref{H}) and (\ref{Kalpha}), while the inequalities (\ref{cool2}) follow from (\ref{cool1}) and (\ref{Kseminorm}).
\end{proof}

\vspace{5mm}
Now we prove our main result which is an infinite dimensional generalization and extension of the main result from  \cite{S93}, which was a crucial step in the proof of the existence of Nash equilibria in the model of
an exchange economy with complete markets studied in \cite{SY89}.
For a substochastic kernel operator $S$ on $L$ we define
$$C(S)=\{f\in L_+: f \gg 0 \;\;\mathrm{and \;\; there \;\; exists \;\; a \;\; stochastic}\;\; A\ge S \;\;\mathrm{such \;\; that}\;\; Af=f \}.$$
 The meaning of the set $C(S)$ in the economic setting of strategic market games is explained in Remark \ref{Aleluja}(ii).

\begin{theorem}
Let $L \subset L^1(X,\mu)$ be an ideal, such that $X$ is the carrier of $L$. Assume that $S$ a substochastic kernel operator on L, which is not a stochastic operator and let  $\alpha _i \ge 0$, $i=1,\dots, m$, 
such that $\sum _{i=1} ^m \alpha _i =1$. Then it holds:

\noindent (i) $C(S)=\{f\in L_+: f \gg 0, Sf\le f \}$;

\noindent (ii) $C(S)$ is a wedge;

\noindent (iii) $f_1 ^{\alpha _1}f_2 ^{\alpha _2} \cdots f_m ^{\alpha _m} \in C(S)$ if $f_1,f_2, \dots, f_m \in C(S)$, i.e. $C(S)$ is a logarithmic convex set.

\noindent (iv) If $f_1,f_2, \dots, f_m \in C(S)$, then
\be 
\|S(f_1 ^{\alpha _1}f_2 ^{\alpha _2} \cdots f_m ^{\alpha _m})\|_L \le \|f_1 ^{\alpha _1}f_2 ^{\alpha _2} \cdots f_m ^{\alpha _m}\|_L \le \|f_1 \|_L ^{\alpha _1}\|f_2 \|_L ^{\alpha _2} \cdots \|f_m\|_L ^{\alpha _m} 
\ee
for any lattice seminorm on $\|\cdot\|$ on $L$.

\noindent (v) If $f_i , g_i \in C(S)$ for $i=1,\ldots ,m$ and $\alpha \in (0,1)$, then 
$$\|S(f_1 ^{\alpha}g_1 ^{1-\alpha} + \ldots +f_m ^{\alpha}g_m ^{1-\alpha}) \|_L \le \|(f_1+ \ldots + f_m)^{\alpha}(g_1+ \ldots + g_m)^{1-\alpha}\|_L $$
\be
\;\;\;  \;\;\;  \;\;\;  \;\;\;  \;\;\;  \;\;\;  \;\;\;  \;\;\;  \;\;\;  \;\;\;  \;\;\;  \;\;\;  \;\;\;  \;\;\;  \;\;\;\le  \|f_1+ \ldots + f_m\|_L ^{\alpha} \cdot \|g_1+ \ldots + g_m\|_L ^{1-\alpha} 
\ee
for any lattice seminorm on $\|\cdot\|_L $ on $L$.
\label{sahi}
\end{theorem}
\begin{proof}
(i) Let $f\in L_+$ such that $f\gg 0$. We need to prove that $f \in C(S)$ if and only if $Sf \le f$.

If $f\in C(S)$, then $f=Af\ge Sf$.

For the converse let us assume that $Sf \le f$. Let $\varphi := f-Sf \in L_+$ and $\psi =1 - s$, where $s(y)=\int _X s(x,y)\;dx$ for almost all $y\in X$. Then 
$\lambda := \int _X \psi(y)f(y)dy >0$, since $\psi >0$ on the set of positive measure and $\lambda < \infty$, since $f \in L^1 (X, \mu)$. 
Let $A \ge S$ be a positive kernel operator on $L$ with a kernel
$$a(x,y) = s(x,y)+ \frac{1}{\lambda}\varphi(x)\psi(y)$$
for $x,y \in X$. Note that $a(x,y)$ is a $\mu \times \mu$-measurable function
 on $X \times X$ such that, for all $g \in L$ and for almost all $x \in X$,
$\int_X |a(x,y) g(y)| \, dy < \infty$, since $\psi \le 1$, $\varphi < \infty$ almost everywhere and $g \in L^1(X,\mu)$.

For almost all $x \in X$ we have 
$$(Af)(x)= \int_X s(x,y) f(y) \, dy + \frac{1}{\lambda}\varphi(x) \int _{X}\psi(y)f(y) \, dy=
(Sf)(x)+ \varphi(x)=f(x),$$
so $Af=f$. By Fubini's theorem we obtain
$$\lambda = \int _X f(y)\,dy - \int _X s(y)f(y)\,dy = \int _X f(y)\,dy - \int _X \left(\int _X s(x,y)f(y)\, dy \right)\,dx$$
$$=\int _X f(x)\,dx - \int _X \left(Sf \right)(x)\,dx= \int _X \varphi (x) \,dx.$$
Therefore it follows
$$\int _X a(x,y) \, dx = \int _X s(x,y) \, dx + \frac{1}{\lambda}\psi(y) \int _X \varphi(x) \, dx =s(y)+\psi (y)=1$$
for almost all $y \in X$ and so $A$ is a stochastic operator on $L$. So we have proved that \\
$C(S)=\{f\in L_+: f \gg 0, Sf\le f \}$. This also implies (ii).

(iii) Let $f_1,f_2, \dots, f_m \in C(S)$ and $h=f_1 ^{\alpha _1}f_2 ^{\alpha _2} \cdots f_m ^{\alpha _m}$. 
Then we have 
$$ Sh
\le (Sf_1)^{\alpha _1}\cdots (Sf_m)^{\alpha _m}\le f_1^{\alpha _1} \cdots f_m^{\alpha _m}=h$$
by (\ref{Kalpha}) and so $h \in C(S)$, which proves (iii).

The inequalities in (iv) follow from (iii) and (\ref{rho}), while the inequalities in (v) follow from (\ref{cool2}), (ii), (iii) and (\ref{rho}).
\end{proof}
\begin{remark}{\rm  If $\mu (X) <\infty$, then $L_{p} (X,\mu)  \subset L_{1} (X,\mu) $ for all $p \in [1,\infty]$. In this case Theorem \ref{sahi} can be applied to $L=L_{p} (X,\mu)$ equipped  with e.g. standard lattice norms $ \|\cdot\|_p$. Moreover, by \cite[Theorem 6.6 on p.77]{BS88} (see also \cite{CR07}) the same applies to any rearrangement-invariant Banach function space. 
}
\end{remark}

\begin{remark}{\rm If $K$ is a strictly positive operator on $L$ that commutes with a substohastic kernel operator $S$, then $f \in C(S)$ implies $Kf \in C(S)$. Indeed,
$Sf \le f$ implies $SKf=KSf \le Kf$.

Let $\mathcal{A}  _+$ denote the collection of all power series 
$$ F(z)=\sum_{j=0}^\infty\alpha _j z^j $$ 
having nonnegative coeficients $\alpha _j \ge 0$ ($j=0,1,\dots$). 
Let $R_F$ be the radius of convergence of $F\in\mathcal{A}  _+$, that is, we have
$$ \frac{1}{R_F}= \limsup_{j\to\infty} \alpha _j^{1/j} . $$
If $T$ is an operator on a Banach space such that the spectral radius $\rho (T)<R_f$, then the operator $f(T)$ is defined by 
$$ f(T)=\sum_{j=0}^\infty\alpha _j T^j . $$

So if  $L$ is an ideal, which is also a Banach space and if $F\in \mathcal{A} _+$, then $F(S)$ commutes with $S$ and so $f \in C(S)$ implies $F(S)f \in C(S)$, if $F(S)f \gg 0$.
 In particular, by choosing the exponential series and the C. Neumann series for $F\in\mathcal{A} _+$, it follows that

(i) $f\in C(S)$ implies that $\exp (S)f \in C(S)$;

(ii) if $f\in C(S)$ and $\lambda > \rho (S)$, then  $(\lambda I - S)^{-1}f \in C(S)$.}
\label{power}
\end{remark}

\section{Non-negative matrices and applications}

In this final section we  
apply our results to (finite or infinite) non-negative matrices that define (weighted) operators on an ideal $L$ of sequences and  point out possible applications of our results.

 Let $R$ denote the set $\{1, \ldots, n\}$ for some $n \in \NN$ or the set $\NN$ of all natural numbers. 
Let $S(R)$ be the vector lattice of all complex sequences $(x_i)_{i\in R}$ and $L \subset S(R)$ an ideal. In the case when $\mu$ is the counting measure on $R$, then a non-negative matrix $A=[a_{ij}]_{i,j \in R} $ defines a (kernel) operator on $L$ if $Ax \in L$ for all $x\in L$, where
$(Ax)_i = \sum _{j \in R}a_{ij}x_j$.

More generally, throughout this section let $\mu (\{i\})=\omega_i > 0$. Given ideals $L,N \subset S(R)$,  a non-negative matrix $A$ defines a (kernel)  weighted operator $A_{\omega}$ from $L$ to $N$ if $A_{\omega} x \in N$ for all $x\in L$, where
$(A_{\omega} x)_i = \sum _{j \in R}a_{ij}x_j\omega _j$. Note that $A_{\omega}$ is a positive operator, since  $A_{\omega} x \in N_+$ for all $x\in L_+$. Recall that a positive weighted operator $S_{\omega}$ on $L$ is a
(column) substochastic operator, if  $s_j=  \sum _{i \in R} s_{ij}\omega _i \le 1$ for all $j\in R$. The operator $S_{\omega}$ is  stochastic, if $s(y)=1$ for all $j\in R$, and  $S_{\omega}$ is strictly (column) substochastic if $s(y)<1$ for all $j\in R$.   For applications of strictly substochastic operators (matrices) to the Google Page-rank model of internet usage and open Leontief model of an economy we refer the reader to \cite{S10} and the references cited there.

Similarly as in \cite{DP05}, \cite{DP10}, \cite{P12}, \cite{DP16} let us denote by $\mathcal{L}$ the collection of all ideals
$L \subset S(R)$ 
satisfying the property that $e_i = \chi_{\{i\}} \in L$. 
 Note that $S(R)=M_0 (R,\mu)$ and that for $L\in \mathcal{L}$ the set $R$ is the carrier of $L$.
In this setting our Theorem \ref{sahi} says the following.
\begin{corollary}
Let $L \subset l^1 _{\omega}$ such that $L\in \mathcal{L}$. Assume that $S_{\omega }$ is a substochastic weighted operator on $L$, which is not stochastic, and let  $\alpha _i \ge 0$, $i=1,\dots, m$, 
such that $\sum _{i=1} ^m \alpha _i =1$. Then it holds:

\noindent (i) $C(S_{\omega})=\{x \in L_+: x >>0, Sx\le x \}$;

\noindent (ii) $C(S_{\omega })$ is a wedge;

\noindent (iii) $x_1 ^{\alpha _1}x_2 ^{\alpha _2} \cdots x_m ^{\alpha _m} \in C(S_{\omega })$ if $x_1,x_2, \dots, x_m \in C(S_{\omega })$, i.e. $C(S_{\omega })$ is a logarithmic convex set.

\noindent (iv) If $x_1,x_2, \dots, x_m \in C(S_{\omega })$, then
$$\|S _{\omega} (x_1 ^{\alpha _1}x_2 ^{\alpha _2} \cdots x_m ^{\alpha _m})\|_L \le \|x_1 ^{\alpha _1}x_2 ^{\alpha _2} \cdots x_m ^{\alpha _m}\|_L \le \|x_1 \|_L ^{\alpha _1}\|x_2 \|_L ^{\alpha _2} \cdots \|x_m\|_L ^{\alpha _m} $$
for any lattice seminorm on $\|\cdot\|$ on $L$.

\noindent (v) If $x_i , y_i \in C(S_{\omega })$ for $i=1,\ldots ,m$ and $\alpha \in (0,1)$, then 
$$\|S_{\omega}(x_1 ^{\alpha}y_1 ^{1-\alpha} + \ldots +x_m ^{\alpha}y_m ^{1-\alpha}) \|_L \le \|(x_1+ \ldots + x_m)^{\alpha}(y_1+ \ldots + y_m)^{1-\alpha}\|_L $$
$$\;\;\;  \;\;\;  \;\;\;  \;\;\;  \;\;\;  \;\;\;  \;\;\;  \;\;\;  \;\;\;  \;\;\;  \;\;\;  \;\;\;  \;\;\;  \;\;\; \;\le  \|x_1+ \ldots + x_m\|_L ^{\alpha} \cdot \|y_1+ \ldots + y_m\|_L ^{1-\alpha} $$
for any lattice seminorm on $\|\cdot\|_L $ on $L$.
\label{sahi_weight}
\end{corollary}
In the following remark we discuss possible applications of our results.
\begin{remark}
\label{Aleluja}
 {\rm (i) {\bf Application 1.} In the well-known Arrow-Debreu model from the mathematical economy (see e.g. \cite{AB03}, \cite{ABB90},  \cite{D12}  and the references cited there) various commodities are exchanged, produced and consumed. The classical model deals with finitely many commodities, but as pointed out in \cite{D12}, already Debreu proposed to study commodity spaces with an infinite number of commodities. Commodities may be understood as physical goods which may differ on the location or on the time that they are produced or consumed, or on the state of the world in which they become available. If we allow an infinite variation in any of this contingencies, then it is natural to consider economies with infinite number of commodities. 

 Let us suppose that there are countably many commodities, and that the commodity space is a positive cone $L_+$ of an ideal $L \subset S(R)$.
A vector $x = (x_1, x_2, x_3, \ldots ) \in L_+$ represents a {\it commodity bundle}, that is, the
number $x_i$ is the amount of the $i$-th commodity. 
If $\omega_i$  is the price for
one unit of the $i$-th commodity, then we can introduce the price vector
 $\omega = (\omega_1, \omega _2,  \omega _3, . . .)$. If the price vector $\omega$ is fixed for all $x\in L_+$ (as is for example in the case of state-determined prices in some EU countries for certain goods, such as gas, or prices for official or some hospital services, etc.)
 and $L\subset l^{1} _{\omega}(R)$, then the value of a commodity bundle $x$ at price $\omega$ is equal to $\|x\|_{l^1  _{\omega}} =\sum _{i \in R} x_i \omega _i$ and $\|x\|_{l^{\infty}  _{\omega}}= \sup_{i \in R} x_i \omega _i = \max_{i \in R} x_i \omega _i $  equals the largest value of the commodities in $x$ at price $\omega$.
 If $S_{\omega }$ is a substochastic weighted operator on $L$,  which is not stochastic, then by Corollary \ref{sahi_weight}(iii) the wedge $C(S_{\omega})$ is a logarithmic convex set. Inequalities in Propositions \ref{Riesz}, \ref{alfe} an \ref{+alpha} and
 Corollary \ref{sahi_weight}(iv)  and (v) provide bounds  for values of considered commodity bundles. In the described setting the number $x_1 ^{\alpha _1}(k)x_2 ^{\alpha _2}(k) \cdots x_m ^{\alpha _m}(k)$,  
for $\alpha _i > 0$, $i=1, \ldots, m$, such that $\sum _{i=1} ^m \alpha _i = 1$, could represent the customers taste or {\it preference} for the $k$th commodity at $m$ different suppliers, depending on the quality of the service eventhough the price for these services is the same at each supplier (e.g. quality of gas, quality of health services, feel-good influence on the customers, etc.). Therefore $ \|x_1 ^{\alpha _1}x_2 ^{\alpha _2} \cdots x_m ^{\alpha _m}\|_{l^1  _{\omega}}$ and  
$ \|x_1 ^{\alpha _1}x_2 ^{\alpha _2} \cdots x_m ^{\alpha _m}\|_{l^{\infty}  _{\omega}}$ represent the value  of the  preference vector $x_1 ^{\alpha _1}x_2 ^{\alpha _2} \cdots x_m ^{\alpha _m}$ at price $\omega$ and the largest value of the commodities in the preference vector, respectively.

(ii) {\bf Application 2.}  Let us describe what kind of role plays our Theorem \ref{sahi}(i) and (iii) (and in particular Corollary \ref{sahi_weight}(i) and (iii)) in the theory of  strategic market games  (for details see \cite{SY89}, \cite{ASSY90}). These results generalize 
results of \cite[Theorem at p.1035 and Lemma at p.1036]{S93} to the infinite dimensional setting. Moreover, \cite[Theorem at p.1035 and Lemma at p.1036]{S93} are merely mathematically formal versions of \cite[Lemma 3 and Lemma 4 (i)]{SY89}. 
In \cite{SY89} an exchange economy with complete markets is described and a general theorem for the existence of active Nash equilibria is proved (\cite[Theorem 1]{SY89}). It is further shown in \cite{SY89} that under replication of traders, these equilibria approach the competitive equilibria of the economy. This model was first proposed by L. Shapley and it represents one of the two possible generalitations (for the second see \cite{ASSY90})  of the "single money" model described  by Dubey and Shubik. It has a pleasant feature that it yields consistent prices. 

 The wedge $C(S)$ from Theorem \ref{sahi} corresponds to the set of all possible positive multiples of price vectors that arise as trader $\alpha$ varies his bid in his strategic set, where a trader $\alpha$ is the only trader with unfixed bids.
As pointed out in \cite{SY89},  \cite[Lemma 3]{SY89}  "is more or less the heart of the argument" (together with the Kakutani's fixed point theorem) in the proof of the existence of Nash equilibria. 
Therefore, our Theorem 2.4, together with a suitable fixed point theorem (see e.g. \cite{BBM14} and the references cited there for some further extensions of Kakutani's fixed point theorem),  is expected to be a key ingredient in the development of the infinite dimensional version of this exchange economy.

(iii) {\bf Applications 3 and 4.} Assume that $S_{\omega}$ is a substochastic weighted operator on ideal $L \subset l^1  _{\omega}$, which is also a Banach space. By Remark \ref{power}(ii)
 $(I - S_{\omega})^{-1}x \in C(S_{\omega})$ if $x\in C(S_{\omega})$. In a special finite dimensional case ($R= \{1, \ldots ,n \}$, $\omega _i =1$ for all $i=1, \ldots, n$, $S_{\omega} = S$) this can be interpreted in the  Google Page-rank model of internet usage and in the  open Leontief model of an economy (see e.g. \cite{S10} and the references cited there)

The open Leontief model of an economy deals with the case of $n$ industries each producing
exactly one good. The production of one unit of good $j$ requires inputs $s_{ij} \ge 0$ of the other goods $i$. Goods 
are measured in “dollars-worth” units, and one usually assumes that every industry runs at a profit,
i.e. it costs less than a dollar to produce one dollar’s worth of any good. This means that the technology
matrix $S=[s_{ij}]$ is strictly (column) substochastic. In order to produce a vector $p = (p_i)$ of goods, the production process consumes $Sp$, leaving only
the excess vector $c = p - Sp$ available for outside use. One thinks of $c$ as a "demand" vector and $p$ as a
"supply" vector, and solving for $p$ in terms of $c$ one gets
$p = (I - S)^{-1} c$.
Since S is strictly column-stochastic the spectral radius of S is less than 1, and $ Y=(I - S)^{-1}$ is a nonnegative
matrix given by  C. Neumann's series. $Y$ is called an impact matrix, since the $ij$th entry of $Y$ is the partial derivative $y_{ij} = \frac{\partial p_i}{\partial c _j}$ and represents the
increase in supply of good $i$ in response to a $1$ unit increase in the demand of good $j$.

The Google Page-rank model  involves a discrete Markovian birth-death process that is specified by a nonnegative
vector $x=(x_i)$ and a non-negative matrix $ S= [s_{ij}]$. Here $x_i$ represents the number of births (initial
visits) per unit time in site $i$, and $s_{ij}$ is the transition probability from site $j$ to site $i$. The matrix $S$
is strictly (column) substochastic, since there is a positive probability of death (logging off). The
steady state vector $p$ (Page-rank) satisfies $p = x + Sp$, whence we get
$p=(I - S)^{-1}x $.
}
\end{remark}

\baselineskip 5mm

\noindent {\bf Acknowledgements.}  The author thanks  Marko Kandi\'{c}, Roman Drnov\v{s}ek and Helena \v{S}migoc  for reading the early version of this paper and for their comments that improved its presentation.

This work was supported in part by grant P1-0222 of the Slovenian Research Agency and
by the JESH grant of the Austrian Academy of Sciences.

\vspace{2mm}

\noindent
\noindent
Aljo\v sa Peperko, \\
Faculty of Mechanical Engineering, University of Ljubljana \\
A\v{s}ker\v{c}eva 6, SI-1000 Ljubljana, Slovenia {\it and} \\
\\
Institute of Mathematics, Physics and Mechanics \\
Jadranska 19, SI-1000 Ljubljana, Slovenia \\
e-mails : aljosa.peperko@fs.uni-lj.si, aljosa.peperko@fmf.uni-lj.si

\end{document}